\documentclass[11pt]{amsart}
\usepackage[letterpaper,margin=1.2in]{geometry}
\usepackage{amsbsy,amsfonts,amsmath,amssymb,amscd,amsthm,mathrsfs}
\usepackage{subfigure,enumitem,graphicx,epsfig,color}

\usepackage[colorlinks=true,linkcolor=black,citecolor=black]{hyperref}
\normalsize
\newtheorem{mainthm}{Theorem}

\newtheorem{theorem}{Theorem}[section]
\newtheorem{lemma}[theorem]{Lemma}

\newtheorem{definition}[theorem]{Definition}
\newtheorem{corollary}[theorem]{Corollary}

\newtheorem{problem}[theorem]{Problem}

\newtheorem{proposition}[theorem]{Proposition}

\newtheorem*{theorem A}{Theorem A}
\newtheorem*{corollary C}{Corollary C}
\newtheorem*{theorem B}{Theorem B}
\theoremstyle{definition}

\begin{document}
\author[A. Fakhari]{Abbaas Fakhari}
\address{\footnotesize \centerline{Department of Mathematics, Shahid
Beheshti University, } \centerline{Tehran 19839, Iran}}
\email{a\_fakhari@sbu.ac.ir}
\title[Connectedness of the set of central Lyapunov exponents]
{Connectedness of the set of central Lyapunov exponents}
\keywords{Homoclinic class, partially
hyperbolic, Lyapunov exponent, $C^1$-generic}
\subjclass[2010]{37B20, 37C29,37C50}
\maketitle
\begin{abstract}
We show that there is a residual subset
$\mathcal{R}$ of $\textsl{Diff}^{\,1}(M)$ such that for any
$f\in\mathcal{R}$ and any partially hyperbolic homoclinic class
$H(p,f)$ with one dimensional center direction, the set of central 
Lyapunov exponents associated with the ergodic with either full support or positive entropy is an
interval. 
\end{abstract}
\section{Introduction}
The main question we try to answer is a reformulation of a fact initiated by Sigmund in a short article (\cite{si}). There, he studied the connectedness of the set of hyperbolic ergodic measure. He established path connectedness of this space in the case of transitive topological Markov shifts and as a corollary, of transitive Axiom Adiffeomorphisms. Since Sigmund's work the interest to this problem has been lost. 

By definition, hyperbolic dynamical systems have nonzero Lyapunovexponents. However, they are not generic in the space of all
dynamical systems. This necessitated weakening the notion of hyperbolicity. A wider class is that of partially hyperbolic systems
studied by many authors (for a complete review, see \cite{bp,hp,hhu}). A partially hyperbolic set may admit zero Lyapunov exponents along
their central directions. As a recent efforts, in Sigmund direction, in \cite{gt} raised a question of whether the space of ergodic measures invariant under some partially hyperbolic systems is path connected. Our main question here is: 

\bigskip
\textbf{Question.}~Is the set of central Lyapunov vectors associated
with the ergodic measures on a partially hyperbolic set convex?

\bigskip
In general, the answer is negative. In \cite{lor}, the authors have
studied destroying horseshoes via heterodimensional cycles. These
maps are partially hyperbolic with a one dimensional center
direction. They proved that every ergodic measure is hyperbolic, but
the set of Lyapunov exponents in the central direction has a  gap.
As known, the existence of heterodimensional cycles isn't a generic
phenomenon. There are examples of open sets of partially hyperbolic
maps whose central Lyapunov exponents contains a convex set. Now, a
question may be arisen:
\begin{center}
{\textbf{\it what can be said in generic mode?}}
\end{center}
There are partial answers to the question in the case of generic
``locally maximal" or ``Lyapunov stable" homoclinic classes
\cite{abcdw,csy}. Here, we give a positive answer to the
question in the case of one dimensional central direction. 

Let $M$ be a compact boundless Riemannian manifold and $f$ be a
diffeomorphism on it. For a periodic point $p$ of $f$, we denote by
$\pi(p)$ the period of $p$. For a hyperbolic periodic point $p$ of
$f$ of period $\pi(p)$ the sets
\begin{center}
 $W^s(p)=\{x\in M : f^{\pi(p)n}(x)\to p  ~\textit {as}~ n\to\infty\},~\mbox{and }$\\
  \hspace{-1cm}$W^u(p)=\{x\in M : f^{-\pi(p)n}(x)\to p ~\textit{ as}~
  n\to\infty\};$
\end{center}
are injectively immersed submanifolds of $M$. A point $x\in
W^s(p)\,\cap\, W^u(p)$ is called a {\it homoclinic point} of $f$
associated with $p$, and it is called a {\it transversal homoclinic
point} of $f$ if the above intersection is transversal at $x$. The
closure of the transversal homoclinic points of $f$ associated with
the orbit of $p$ is called the {\it homoclinic class} of $p$, and  denoted by $H(p,f)$.
 Two saddles $p$ and $q$ are called {\it homoclinically related}, and
write $p\sim q$, if
$$W^s(\mathcal{O}(p))\,{\overline{\pitchfork}}\,W^u(\mathcal{O}(q))\not=\phi\,\,\,\mbox{
and}\,\,\,W^u(\mathcal{O}(p))\,{\overline{\pitchfork}}\,W^s(\mathcal{O}(q))\not=\phi
.$$ Note that, by {\it Smale's transverse homoclinic point theorem}
, $H(p,f)$ coincides with the closure of the set of all hyperbolic
periodic points $q$ with $p\sim q$.
\begin{definition}
A compact invariant set $\Lambda$ is {\it partially hyperbolic} if
it admits a splitting $T_\Lambda M=E^{ss}\oplus E^c\oplus E^{uu}$
such that
\[\|Df|_{E^{ss}}\|<\|Df|_{E^{c}}\|<\|Df|_{E^{uu}}\|,\]
furthermore, $E^{ss}$ is uniformly contraction and $E^{uu}$ is
uniformly expansion; it means
\[\|Df|_{E^{ss}}\|<1<\|Df|_{E^{uu}}\|.\]
\end{definition}
The bundles $E^{ss}$ and $E^{uu}$ of a partially hyperbolic set are
always uniquely integrable. Denote by $W^{ss}_{loc}(x)$ (resp.
$W^{uu}_{loc}(x)$) the {\it local strong stable} (resp. {\it
unstable}) manifold of $x$ tangent to $E^{ss}(x)$ (resp.
$E^{uu}(x)$) at $x$ (\cite{bp,hp,hhu}). Recall that if a diffeomorphism $f$ has
a partially hyperbolic set $\Lambda$ then there is an open
neighborhood $U$ of $\Lambda$ and a neighborhood $\mathcal{U}$ of
$f$ such that the maximal invariant set $\tilde{\Lambda}_g$ of a
diffeomorphism $g\in\mathcal{U}$ which is entirely contained in $U$
is partially hyperbolic.

Denote by $\mathcal{M}_f (M)$ the set of $f$-invariant probability
measures endowed with its usual topology; i.e., the unique
metrizable topology such that $\mu_n\to\mu$ if and only if $\int
f\,d\mu_n\to\int f\,d\mu$ for every continuous function
$f:M\to\mathbb{R}$. The set of all ergodic elements of
$\mathcal{M}_f(M)$ supported on an invariant set $\Lambda$ is
denoted by $\mathcal{M}_e(f|_\Lambda)$.

A point $x$ in a partially hyperbolic set $\Lambda$ is called {\it
regular} if the following limit exists.
\[\lim_{n\to\infty}\frac{1}{n}\log\|Df^n(x)|_{E^c(x)}\|.\]
By Oseledec's theorem, given any invariant measure $\mu$ the
limits exist $\mu$-a.e.. The limits above are called {\it central
Lyapunov exponent} of $x$. When $\mu$ is ergodic the limit doesn't
depend on initial point $\mu$-a.e.. We recall that
an invariant set $\Lambda$ of a diffeomorphism $f$ is {\it Lyapunov
stable} if for any neighborhood $U$ of $\Lambda$ there is a $V$ of
it such that $f^ n(V )\subset U$, for any $n\geq 0$. Surely, for any
hyperbolic periodic point $p$ in a Lyapunov stable invariant set
$\Lambda$ we have $W^u(p)\subset \Lambda$. The set $\Lambda$ is {\it
bi-Lyapunov stable} if it is Lyapunov stable for $f$ and $f^{-1}$.

Here, we try to build ergodic measure with middle Lyapunov exponents. There are two distinct strategies, approximation with periodic measures and flip-flop phenomenon. Although, ergodic measures constructed by the method of periodic
approximations may have full support, they are highly "repetitive" and are likely to have zero entropy. For constructing measure of positive entropy we use the flip-flop family introduced in \cite{bbd}.  A disadvantage of
this method is that the nature of the obtained measure prevents the measures to have full
support on the homoclinic class. In our main theorem, we benefit both method to build the desired measures. 
\begin{mainthm}
For a $C^1$-generic diffeomorphism $f$ and any
partially hyperbolic homoclinic class $H(p,f)$ of $f$ with the
splitting $E^s\oplus E^c\oplus E^u$ such that
\begin{itemize}
\item either $\textrm{dim}(E^c)=1$ or
\item $H(p,f)$ is bi-Lyapunov stable and $E^c$ has a dominated splitting $E^c_1\oplus E^c_2$ into one
dimensional bundles.
\end{itemize}
The central Lyapunov exponents associated with the ergodic measures
with support equal to $H(p,f)$ form a connected set. Furthermore, if $H(p,f)$ is bi-Lyapunov stable and $\textrm{dim}(E^1)=1$ then the set of central Lyapunov exponents associated with the ergodic measures supported on $H(p,)$ and have positive topological entropy forms a connected set.
\end{mainthm}
In fact, following \cite{bbd}, we prove that for any middle $\alpha$ there is an invariant set $\mathbf{K}_\alpha$ containing in the unstable lamination such that all points belong to it have Lyapunov exponents equal to $\alpha$. In particular, any ergodic measure supported on $\mathbf{K}_\alpha$ has Lyapunov exponent $\alpha$.
\begin{problem}
In general, is the set of Lyapunov vectors
associated with the set of ergodic measures on a generic homoclinic
class convex?
\end{problem}
Section 2, describes two examples concerning our main question.
The first example describe an open set of diffeomorphisms
whose center Lyapunov exponents contains an interval. The second example
provides a skew-product for which there is a gap in the center Lyapunov exponents.
Both constructions benefit a realization of a skew-product over a finite generators.
In section 3 we deal with the problem of approximating of ergodic measure by periodic ones. Proof of the Theorem A handled in sections 4 and 5.
\section{Robust and Non-generic Examples}
We need a bit of notations. Consider and iterated function system
$\textrm{IFS}(f_1,\dots,f_k)$ generated by continuous maps
$f_1,\dots,f_k$ on $X$. A sequence of iterates
$x_{n+1}=f_{\omega_n}(x_n)$ with $\omega_n \in \{1,\dots,k\}$ chosen
randomly and independently is called \emph{random orbit} starting at
$x_0=x$. The sequence of compositions $
f_{\omega}^n=f_{\omega_n}\circ\dots \circ f_{\omega_1}$ is the
\emph{orbital branch} corresponding to
$\omega=\omega_1\omega_2\dots\in\Sigma_k^+=\{1,\dots,k\}^\mathbb{N}$.
One can associate with the IFS a step skew-product $S:\Sigma_k^+\times
M\to\Sigma_k^+\times M$ given by
$S(\omega,x)=(\sigma(\omega),f_{\omega_0}(x))$. Considering a Markov
partition with $k$ element of a horseshoe in $\mathbb{S}^2$, one can
realize the skew-product as a diffeomorphism on $\mathbb{S}^2\times
M$ with a partially hyperbolic locally maximal invariant set
homeomorphic to $\Sigma_k^+\times M$.
\subsection{A Robust Example} Suppose that $f_0$ is a
Morse-Smale diffeomorphism on $\mathbb{S}^1$ with exactly one sink
and one source and $f_1$ is an irrational rotation, also assume that
both of them are sufficiently closed to  the identity map in the
$C^1$-topology. It is known that
\begin{itemize}
\item the corresponding step skew-product is robustly transitive on the
invariant set $\Lambda=\Sigma_2^+\times\mathbb{S}^1$, in the sense
that any perturbation of the initial skew-product is transitive on a
locally maximal invariant set $\tilde{\Lambda}$ homemorphic to
$\Sigma_2^+\times\mathbb{S}^1$.
\item $\tilde{\Lambda}$ is a local homoclinic class.
\end{itemize}
The conjugacy sends the central foliation of $S$ to the central
foliation of the perturbed diffeomorphism which are all circle.
In fact, Gorodestki and Ilyashenko in \cite{gi} has proved the
H\"{o}lder continuity of the center manifolds of the perturbed
diffeomorphisms (see also Theorem B in \cite{vy}). Using this key
property, the authors in \cite{kn} have proved the existence of an
ergodic measure with full support whose center Lyapunov exponent is
zero. The same arguments can be adopted (with a duplicated proof) to
fill an interval with the center Lyapunov exponents associated with
the different ergodic measure. Furthermore, the ergodic measure may
be chosen in such a way that their support equal to the whole of
invariant set.
\begin{theorem}
For any perturbation of $S$, there is a non-degenerate interval $I$
containing in the set of central Lyapunov exponents associated with
the ergodic measure with full support.
\end{theorem}
\subsection{A Non-generic Example with a Center Gap} The crucial
step of the example below (Lemma~\ref{lemcontraction}) borrowed from
\cite{lor}. Let $f_0:[0,1]\to[0,1]$ be a map satisfying the following
properties
\begin{itemize}
\item $f_0(0)=0, f_0(1)=1$ and  $f_0(x)>x$, for any $x\in(0,1)$;
\item $f_0^\prime$ is decreasing,
$f_0^\prime (1)=1/\alpha$ and $f_0^\prime (0)=\alpha$, for some
$\alpha>1$;
\end{itemize}
and for some $\beta>\alpha$, $f_1(x)=\beta(1-x)$ (see Figure \ref{F1}).
\begin{figure}[h]\label{F1}
\[\includegraphics[scale=1]{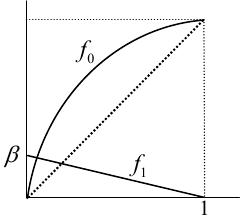}
\]
\caption{The graph of $f_0$ and $f_1$}
\label{F1}
\end{figure}
\begin{lemma}
\label{lemcontraction} Let $\omega\in\Sigma_2^{+}$ and $n_i$ be the
position of $i$-th 1 in $\omega$. There are $0<\lambda<1$ and
$C(\omega)$ such that for any $x\in[0,1]$,
$$|Df^{n_i}_\omega(x)|\leq C(\omega)\lambda^i.$$
\end{lemma}
Putting $\bar{0}=(0,0,\ldots)\in\Sigma_2^{+}$, the skew-product has
two fixed points $p=(\bar{0},0)$ and $q=(\bar{0},1)$. Suppose that
$\mu$ is an ergodic measure for the skew-product.
Then $\pi^*\mu$ is an ergodic measure for the shift map. If
$\pi^*\mu(\bar{0})=1$ then $\mu$ is $\delta_p$ or $\delta_q$
and hence has negative or positive center lyapunov exponents. If
$\pi^*\mu(\bar{0})<1$ then for $\mu$-almost every
$(\omega,x)\in \Sigma_2^{+}\times M$, 1 appears with positive
frequency $\gamma$ in $\omega$. By the lemma above,
$$\liminf \frac{1}{n_i}\log |Df^{n_i}_\omega(x)|\leq \liminf \frac{1}{n_i}\log C\lambda^i=\gamma\log\lambda<0.$$
that is, any ergodic measure is hyperbolic and there are ergodic
measures with positive and negative center Lyapunov exponents.
\begin{theorem}
For the skew-product generated by IFS$(f_0,f_1)$, any ergodic
measure is hyperbolic and there are ergodic measures whose center
Lyapunov exponents have different sign. In particular, the set of
center Lyapunov exponents has a gap.
\end{theorem}
\section{Ergodic Measure Approximated by Periodic Ones}
First, we state two following lemmas to guarantee the existence of
hyperbolic periodic points with desired central Lyapunov exponents
inside the homoclinic class.
\begin{lemma}(Pliss's Lemma \cite{pl})
Given $\lambda_0<\lambda_1$ and $\mathcal{O}(p)\subset
\tilde{\Lambda}$ there is some natural number $m$ such that if
$\prod_{j=0}^{t} \|Df|_{E^c(f^j(p_n))}\|\leq \lambda_0^t$ for some
$t\geq m$ then there is a sequence $n_1<n_2<\cdots<n_k\leq t$ such
that $\prod_{j=n_r}^{t} \|Df|_{E^c(f^j(p_n))}\|\leq
\lambda_1^{t-n_r}$, for any $1\leq r\leq k$.
\end{lemma}
\begin{lemma}(\cite{ps})\label{stable1}
For any $0<\lambda<1$, there is $\epsilon>0$ such that for $x\in
\tilde{\Lambda}$ which satisfies $\prod_{j=0}^{n-1}
\|Df|_{E^c(f^j(p_n))}\|\leq \lambda^n$, for all $n>0$, we have
$\textrm{diam}(f^n(W^{cs}_\epsilon(x))\to 0$. In other word, the
center stable manifold of $x$ with size $\epsilon$ is in fact stable
manifold.
\end{lemma}
To prove Theorem A, we need to recall some $C^1$ generic statements.
For a partially hyperbolic homoclinic class $H(p,f)$ and
$\mu\in\mathcal{M}_e(f|_{H(p,f)})$, Let $\lambda^c_\mu$ be the
Lyapunov exponent of $\mu$ along the central direction. For a
hyperbolic periodic point $q\in H(p,f)$ of period $\pi(q)$,
$\lambda^c_q=\log|\nu|/\pi(q)$ where $\nu=Df^{\pi(q)}(q)|_{E^c(q)}$.
Put
$$\mathcal{LE}^c(H(p,f))=\overline{\{\lambda^c_q;\,q\in H(p,f)\cap Per(f)\}}$$
\begin{proposition}\label{generic1}
There is a residual subset $\mathcal{R}_0$ of $\textsl{Diff}^{\,1}(M)$
such that for any $f\in\mathcal{R}$ and any hyperbolic periodic
point $p$ of $f$,
\begin{itemize}
\item [(1)] $f$ is Kupka-Smale,
\item [(2)] the set $\mathcal{LE}^c(H(p,f))$ is connected (\cite{abcdw}),
\item [(3)] for any periodic point $q\in H(p,f)$ with
$\textrm{index}(q)=\textrm{index}(p)$, we have $q\sim p$,
(\cite{csy,fa}),
\item [(4)] for any ergodic measure $\mu$ of $f$, there is a sequence of
periodic point $p_n$ such that $\mu_{p_n}\to \mu$ in weak$^*$
topology and $\mathcal{O}(p_n)\to \textrm{Supp}(\mu)$ in Housdorff
metric (\cite{abc}).
\item [(5)] If a homoclinic class $H(p,f)$ is bi-Lyapunov stable
then for any sufficiently close $g$, $H(p_g,g)$ is bi-Lyapunov
stable, where $p_g$ is the continuation of $p$.
\end{itemize}
\end{proposition}
Using arguments above we prove that any hyperbolic ergodic measure
supported on a $C^1$-generic partially hyperbolic homoclinic class
$H(p,f)$ can be approximated by periodic measures also supported on
the homoclinic class. In fact, this is a special case of Bonatti's
conjecture.
%

Lemma \ref{center1}, below, gives a positive answer to the conjecture in the
case of hyperbolic measures on a partially hyperbolic homoclinic
class.
\begin{lemma}\label{center1}
For any $f\in\mathcal{R}_0$, any partially hyperbolic homoclinic
class $H(p,f)$, any hyperbolic measure
$\mu\in\mathcal{M}_e(f|_{H(p,f)})$ and any $\epsilon>0$, there is a
periodic point $q\in H(p,f)$ such that for any one dimensional
center bundle $E^c$,
\[|\lambda^c_p-\lambda^c_\mu|<\epsilon.\]
\end{lemma}
\begin{proof}
Choose $\lambda_0=\lambda^c_\mu+\epsilon<\lambda_1<0$. By
Proposition \ref{generic1}, for any $\mu\in\mathcal{M}_e(f|_{H(p,f)})$, there
is a sequence $\{p_n\}\subset \tilde{\Lambda}$ such that
$\lambda^c_{p_n}\to\lambda^c_\mu$. Hence, for sufficiently large
$n$, $\lambda^c_{p_n}<\lambda_0=\lambda^c_\mu+\epsilon<0$ and so
$\prod_{j=0}^{\pi(p_n)}|Df|_{E^c(f^j(p_n))}|\leq
(\exp^{\lambda_0})^{\pi(p_n)}.$ Now, by Pliss Lemma, for large $n$,
there is a point $q_n\in\mathcal{O}(p_n)$ such that
$\prod_{j=0}^{t}|Df|_{E^c(f^j(q_n))}|\leq (\exp^{\lambda_1})^t$ for
any $t\leq \pi(p_n)$. By Lemma \ref{stable1}, the center-stable manifolds of
$q_n$'\,s are in fact stable manifolds and have uniform size. This
implies that for large $n,m\in\mathbb{N}$, $p_n\sim p_m$ and $p_n\in
H(p,f)$.
\end{proof}
In what follows, we first state a lemma which is a modification of
some recent results on $C^1$-generic dynamics in our context
(\cite{csy}). The second part of the Lemma helps us to approximate ergodic
measures, not necessarily hyperbolic, supported on bi-Lyapunov
stable homoclinic class by periodic ones, Theorem \ref{generic2} below. We recall that a
periodic point $p$ of a diffeomorphism $f$ has {\it weak Lyapunov
exponent} along a center bundle $E^c$ if for some small $\delta>0$,
$|\lambda^c|<\delta.$.
\begin{lemma}\label{generic3}
There is a residual subset $\mathcal{R}_1$ of $\textsl{Diff}^{\,1}(M)$
such that for any $f\in\mathcal{R}_1$ and any homoclinic class
$H(p,f)$ of $f$ having a partially hyperbolic splitting $E^s\oplus
E^c\oplus E^u$ with non-hyperbolic center bundle $E^c$, we have
\begin{itemize}
\item if $E^c$ is one-dimensional then there is a periodic
point in the homoclinic class with weak Lyapunov exponent along
$E^c$,
\item if $E^c$ has a dominated splitting $E^c=E^c_1\oplus E^c_2$ into
one dimensional subbundles then either $H(p,f)$ contains a periodic
point whose index is equal to $\textrm{dim}(E^s)$ or it contains a
periodic point with weak Lyapunov exponent along $E^c_1$.
\end{itemize}
\end{lemma}
\begin{theorem}\label{generic2}For $C^1$-generic bi-Lyapunov stable homoclinic class
$H(p,f)$, satisfying Proposition 2.3 and Lemma 2.5, any ergodic
measure can be approximate by periodic measure inside the homoclinic
class.
\end{theorem}
\begin{proof}
As before, if an ergodic measure is hyperbolic then nothing remains.
Hence, suppose that $\mu$ is a non-hyperbolic ergodic measure
supported on the homoclinic class $H(p,f)$. Using the forth part of
Proposition \ref{generic1}, one can find a sequence $\{p_n\}$ of periodic point
such that $\mathcal{O}(p_n)\to$ Supp$(\mu)$ in Housdorff metric and
also the periodic measures associated with $p_n$ tends toward $\mu$ in
weak topology. Now, if the homoclinic class contains a periodic
point $q$ of index $s=\textrm{dim}(E^s)$ then by the partial
hyperbolicity, for large $n$, and positive iteretes $m_n$,
$W^u(q)\cap W^s(f^{m_n}(p_n))\neq\emptyset$. Hence, by the Lyapunov
stability, $p_n\in H(p,f)$. Thus, suppose that for any periodic
point in the homoclinic class its index equals to $s+1$.

By Lemma \ref{generic3}, there is a periodic point $q$ in the homoclinic class
with weak Lyapunov exponents along the central direction $E^c_1$. We
note that by the assumption
$\textrm{index}(q)=\textrm{index}(p)=s+1$ and so $q\sim p$. Now,
using Gourmelon's perturbation lemma \cite{ga}, which allows us to
perturb the derivative controlling the invariant manifolds, we can
produce a periodic point $q^\prime$ of index $s$ such that
$W^s(p)\cap W^u(q^\prime)\neq\emptyset$. Again, by the Lyapunov
stability, $q^\prime\in H(p,f)$. This is contradiction.
\end{proof}
\subsection{Remarks in the case $E^c=E^c_1\oplus E^c_2$}
In general, we don't know much about the topological properties of
the set of ergodic measures supported on a homoclinic class. Since
the central directions are one dimensional the Lyapunov exponents
along them define a continues map with in weak-topology. As before,
one can deduce that the closure of the set of central Lyapunov
exponents associated with the hyperbolic ergodic measures forms a
convex set. One should notice that by the domination, the Lyapunov
exponent along $E^c_1$ is strictly less than the Lyapunov exponents
along $E^c_2$. Hence, in the case of non-hyperbolic ergodic measures
two disjoint cases may be occurred, either
$\lambda^\mu_1<0=\lambda^\mu_2$ or $\lambda^\mu_1=0<\lambda^\mu_2$.
The overall picture of the set of central Lyapunov exponents is
illustrated as the general position in Figure \ref{Fig2}, left.
\begin{figure}[h]
\begin{center}
\includegraphics[scale=.72]{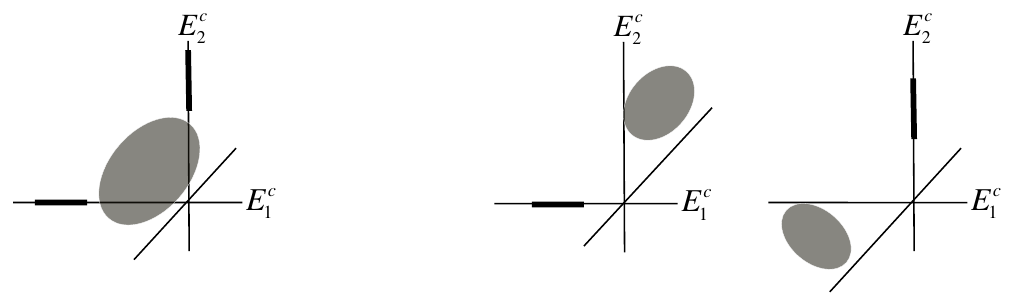}
\end{center}
\caption{Left: General Position, Right: Special Position}
\label{Fig2}
\end{figure}
The shadowed part is the closure of the set of central Lyapunov
exponents associated with the hyperbolic ergodic measures supported on
$H(p,f)$ which is in fact the set $\mathcal{LE}^c(H(p,f))$. Using
the central model for chain transitive sets, S. Crovisier {\it et
al.} have obtained the following version of the Mane's ergodic
closing lemma inside the homoclinic class (\cite{csy}).
\begin{proposition}
For any $C^1$-generic diffeomorphism $f$, let $H(p,f)$ be a
partially hyperbolic homoclinic class with one dimensional central
bundles. Let $i$ be the minimal stable dimension of its periodic
orbits. If $H(p,f)$ contains periodic point with $i^{th}$ weak
Lyapunov exponent, then for any ergodic measure supported on
$H(p,f)$ whose $(i-1)^{th}$ Lyapunov exponent is zero, there exists
periodic orbits contained in $H(p,f)$ whose associated measures
converge for the weak topology towards the measure.
\end{proposition}
The lemma suggests to special situations showed in the Figure \ref{Fig2}, right.
\section{Generating a Measure with Full Support}
Now, we follow the approach suggested in \cite{gikn} to provide
sufficient conditions for the existence of desired ergodic invariant
measure as a limit of periodic ones.
\begin{definition}
A periodic orbit $X$ is a $(\gamma,\chi)$-good approximation of a
periodic orbit $Y$ if the following holds
\begin{itemize}
\item for a subset $\Gamma$ of $X$ and a projection $\rho:\Gamma\to
Y$
\[d(f^j(x),f^j(\rho(x)))<\gamma,\]
for any $x\in \Gamma$ and any $j=0,1,\ldots,\pi(Y)-1$;
\item $card(\Gamma)/card(X)\geq \chi$;
\item $card(\rho^{-1}(y)$ is the same for all $y\in Y$.
\end{itemize}
\end{definition}
The next lemma is a key point in the proof of the ergodicity of a
limit measure. This lemma was suggested by Yu. Ilyashenko
\cite{gikn}. The following modified version, borrowed from
\cite{bdg}, determines the support of the obtained ergodic measure.
\begin{lemma}\label{support}
Let $\{X_n\}$ be a sequence of periodic orbits with increasing
periods $\pi(X_n)$ of $f$. Assume that there are two sequences of
numbers $\{\gamma_n\}$, $\gamma_n>0$ and $\{\chi_n\}$, $\chi_n\in
(0,1]$, such that
\begin{itemize}
\item [1.] for any $n\in\mathbb{N}$ the orbit $X_{n+1}$ is a
$(\gamma_n,\chi_n)$-good approximation of $X_n$;
\item [2.] $\sum _{n=1}^\infty \gamma_n<\infty$;
\item [3.]$\prod_{n=1}^\infty \chi_n\in (0,1]$.
\end{itemize}
Then the sequence $\{\mu_{X_n}\}$ of atomic measures has a limit
$\mu$ which is ergodic and
\begin{equation}\label{support1}
\textrm{Supp}(\mu)=\bigcap_{k=1}^\infty\big(\overline{\bigcup_{l=k}^\infty
X_l}\big)
\end{equation}
\end{lemma}
\begin{lemma}(lemma 3.5 in \cite{bdg})\label{app}
There is a residual subset $\mathcal{R}_2$ of $\textsl{Diff}^{\,1}(M)$
such that for any $f\in\mathcal{R}_2$, any partially hyperbolic
homoclinic class $H(p,f)$, any one dimensional center bundle $E^c$,
any saddle $q$ in $H(p,f)$ and any $\epsilon>0$, $H(p,f)$ contains a
saddle $q_1$ which is homoclinically related to $q$, its orbit is
$\epsilon$-dense in the homoclinic class $H(p,f)$ and
\[|\lambda^c_q-\lambda^c_{q_1}|<\epsilon.\]
\end{lemma}
\subsection{Full Support, First Part of Theorem A}
We first proceed with the first case of the theorem. Let $f\in
\mathcal{R}_0\cap\mathcal{R}_1\cap\mathcal{R}_2$ and $H(p,f)$ be a
partially hyperbolic homoclinic class of $f$ with one dimensional
center direction. Put
\[\lambda^c_{min}=\min
\lambda^c(\mu),\,\,\lambda^c_{max}=\max \lambda^c(\mu),\] where
``$\min$" and ``$\max$" are given over all $f$-invariant ergodic
measures supported on $H(P,f)$. We consider three possible cases.

At the fist consider the case in which $\lambda^c_{min}<0<\lambda^c_{max}$. Let
$\alpha\in(\lambda^c_{min},\lambda^c_{max})$. We use Lemma \ref{support} to
construct an ergodic measure $\mu$ with Supp$(\mu)=H(p,f)$ such that
$\lambda^c_\mu=s$. By Lemma \ref{app}, there are two saddles $p$ and $q_0$
such that $\lambda^c_p<0<\alpha<\lambda^c_{q_0}$. Two cases may be
occurred.

{\bf Case 1.}~$\alpha=0$. This case is deduced from \cite{bdg}. In fact, the authors in
\cite{bdg} have proved that if a $C^1$-generic homoclinic class
which has partial splitting with one dimensional center direction
has periodic points with different index then it can admit a
non-hyperbolic ergodic measure whose support equals to the whole of
the homoclinic class.

{\bf Case 2.}~$\alpha\neq 0$.  Suppose that $\alpha>0$. By the second item of
Proposition \ref{generic1}, there is a periodic point $p_0\in H(p,f)$ such that
$0<\lambda^c_{p_0}<\alpha<\lambda^c_{q_0}$. Inductively, we find two
sequences of $\{p_n\}$ and $\{q_n\}$ of the same index such that
\begin{itemize}
\item [(a)]~$0<\lambda^c_{p_{n-1}}<\lambda^c_{p_n}<\big(s+\lambda^c_{p_{n-1}}\big)/2
<\alpha<\lambda^c_{q_n}<(s+\lambda^c_{q_{n-1}})/2<\lambda^c_{q_{n-1}}$,
\item [(b)]~$q_n$ is $\epsilon/2^n$-dense in $H(p,f)$,
\item [(c)]~$q_n$ is $(\epsilon/2^n,(1-c/2^n))$-good
approximation of $q_{n-1}$, for some positive constant $c$.
\end{itemize}
Let $0<\lambda^c_{p_{n-1}}<\alpha<\lambda^c_{q_{n-1}}$. By Lemma \ref{app},
there is a saddle $p_n\in H(p,f)$ with $\epsilon/2^n$-dense orbit
such that
$0<\lambda^c_{p_{n-1}}<\lambda^c_{p_n}<\big(\alpha+\lambda^c_{p_{n-1}}\big)/2<\alpha$.
Since $p_n$ and $q_{n-1}$ have the same index they are
homoclinically related. Let
$x\in
W^s(\mathcal{O}(q_{n-1}))\,\overline{\pitchfork}\,W^u(\mathcal{O}(p_{n}))$ and $y\in
W^u(\mathcal{O}(q_{n-1}))\,\overline{\pitchfork}\,W^s(\mathcal{O}(p_{n}))$.
Consider a locally maximal hyperbolic set $\Lambda_n$ containing
$\mathcal{O}(p_n)\cup
\mathcal{O}(x)\cup\mathcal{O}(y)\cup\mathcal{O}(q_{n-1})$. For some
suitable $m_n\in\mathbb{N}$, put
\[\mathcal{PO}_n=\{\underbrace{\mathcal{O}(p_n)}_{t_n-\textrm{times}},f^{-m_n}(x),\ldots,f^{m_n}(x),
\underbrace{\mathcal{O}(q_{n-1})}_{s_n-\textrm{times}},f^{-m_n}(y),\ldots,f^{m_n}(y)\},\]
where $2m_nM/\big(t_n\pi(p_n)\lambda^c_{p_n}+s_n
\pi(q_{n-1})\lambda^c_{q_{n-1}}+2mM\big)<\epsilon/2^n$. If
$\mathcal{PO}_n$ is $\epsilon/2^n$-shadowed by a saddle $q_n$ then
$\pi(q_n)=t_n\pi(p_n)+s_n\pi(q_{n-1})+2m_n$ and for suitable $t_n$
and $s_n$ one has
\[\alpha<\lambda^c_{q_n}\thickapprox\frac{t_n\pi(p_n)\lambda^c_{p_n}+s_n
\pi(q_{n-1})\lambda^c_{q_{n-1}}+2m_nM}
{t_n\pi(p_n)+s_n\pi(q_{n-1})+2m_n}<\frac{\alpha+\lambda^c_{q_{n-1}}}{2},\]
where $M=\max_{x\in M}\log\|Df(x)\|$. It is't difficult to see that
\[\chi_n=\frac{s_n\pi(p_n)}{t_n\pi(q_{n-1})+s_n\pi(p_n)+2m}=
1-\frac{t_n\pi(q_{n-1})+2m}{t_n\pi(q_{n-1})+s_m\pi(p_n)+2m}>1-c/2^n,\]
for some positive constant $c$. Now, the sequence $\{q_n\}$
satisfies in the assumption of Lemma \ref{support}. Hence, the weak$^*$ limit
$\mu$ of the atomic measures $\mu_{q_n}$ is ergodic and
\[\lambda^c_\mu=\int |Df|_{E^c}|\,d\mu=\lim_{n\to\infty} \int
|Df|_{E^c}|\,d\mu_{q_n}=\lim_{n\to\infty}\lambda^c_{q_n}=\alpha.\]
Furthermore, by (\ref{support1}), $\textrm{Supp}(\mu)=H(p,f)$.

Now, we are dealing with the remaining cases. First Consider the case $\lambda^c_{min}>0$ or $\lambda^c_{max}<0$. In this case, the
central direction is hyperbolic (\cite{cao}). Another cases that may be happened are $\lambda^c_{min}=0$ or $\lambda^c_{max}=0$. In these cases, by the first part of Lemma \ref{generic3}, the problem is reduced to the previous
ones.
\section{Generating a Measure with Positive Entropy}
As it is seen in the previous section the crucial is the existence of periodic points of different indices. Here, we only deal with this case. This section is essentially owed to \cite{bbd}. We start by stating some preliminarily notions and definitions.
\subsection{Flip-flop Family}
Let $(X,d)$ be a metric space, $f:X\to X$ be a continuous map, $K$ a compact subset of $X$, $\phi:K\to \mathbb{R}$ a continuous
function and $\alpha$ a real number. A flip-flop family with respect to couple $(\phi,\alpha)$ 
is a family $\mathfrak{F}_{(\phi,\alpha)}$ of compact subsets of $K$ with bounded diameter such that $\mathfrak{F}_{(\phi,\alpha)}=\mathfrak{F}^-_{(\phi,\alpha)}\bigcup \mathfrak{F}^+_{(\phi,\alpha)}$ into disjoint families such that
\begin{itemize}
\item [(1)] There exists a constant $\chi_0>0$ 
such that for any $D^+\in\mathfrak{F}^+_{(\phi,\alpha)}$ and $D^-\in\mathfrak{F}^-_{(\phi,\alpha)}$, $$\phi|_{D^-}<\alpha-\chi_0<\alpha+\chi_0<\phi|_{D^+};$$
\item [(2)] For any $D\in\mathfrak{F}_{(\phi,\alpha)}$, there exist two subsets $D^+$ and $D^-$ of $D$ such that $f(D^-)\in\mathfrak{F}^-_{(\phi,\alpha)}$ and $f(D^+)\in\mathfrak{F}^+_{(\phi,\alpha)}$;
\item [(3)] There exists a constant $\lambda>1$ such that if $x,y$ belong to the same element
$D_0$ of $\mathfrak{F}_{(\phi,\alpha)}$ and if $f(x)$ and $f(y)$ belong also to the same element $D_1$ of $\mathfrak{F}_{(\phi,\alpha)}$ then
$d(f(x),f(y))\geq\lambda d(x,y)$.
\end{itemize}
Now, suppose that $H(p,f)$ is a partially hyperbolic homoclinic class and $\text{index}(q)<\text{index}(p)$ for some $q\in H(p,f)$. 
We say that $H(p,f)$ admits a flip-flop family with respect to a continuous function $\phi(x)$ and $\alpha$ if there is a flip-flop family $\mathfrak{F}_{(\phi,\alpha)}=\mathfrak{F}^-_{(\phi,\alpha)}\bigcup \mathfrak{F}^+_{(\phi,\alpha)}$
both contained in the local unstable lamination. Using the notion of blender, the authors in \cite{bbd} proved the following fact. Let $\mathcal{U}$ be an open set of diffeomorphisms such that for any $f\in\mathcal{U}$, there exist two hyperbolic periodic points $p_f,q_f$ with $\text{index}(q)<\text{index}(p)$, depending continuously on $f$ and in the same partially hyperbolic homoclinic class. Then there exists an open and dense subset $\mathcal{V}$ of $\mathcal{U}$ such that for any $f\in\mathcal{V}$ the homoclinic class $H(p_f,f)$ admits a Flop-flop family associated with $(\phi:=\log |Df^N|_{E^c}|, \alpha)$ with $\lambda^c(q)<\alpha<\lambda^c(p)$ (see Figure \ref{fig:2}).
\begin{figure}[h]
\[\includegraphics[scale=.7]{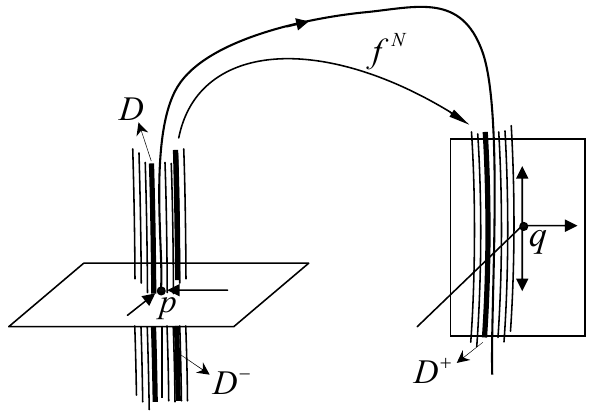}\]
\caption{Flip-flop family associated with $p$ and $q$}\label{fig:2}
\end{figure}
 By a classical generic argument, one can deduce the following generic version.
\begin{theorem}
There is a residual subset $\mathcal{R}$ of \textsl{Diff}$^{\,r}(M)$, $r\geq 1$, such that for any $f\in\mathcal{R}$, any partially hyperbolic homocliinic class $H(p,f)$, any periodic point $q$ of different index of that of $p$ and any $\lambda^c(p)<\alpha<\lambda^c(q)$ there is a constant $N$ and flip-flop family $\mathfrak{F}_{(\phi,\alpha)}=\mathfrak{F}^+_{(\phi,\alpha)}\bigcup \mathfrak{F}^-_{(\phi,\alpha)}$ for $f^N$ and function
$\phi:=\log|Df^N|_{E^c}|$. Furthermore, $\mathfrak{F}^-$ and $\mathfrak{F}^+$ may be chosen in such a way that $\mathfrak{F}^-$ and $\mathfrak{F}^+$ happen respectively in a neighborhood of $p$ and $q$.
\end{theorem}
We want to prove the main result of this section using the characteristic property of Flip-flop family. Let us return to abstract definition above.
\begin{itemize}
\item ($\mathfrak{F}_{(\phi,\alpha)}$-segment). Let $T\in \mathbb{N}^\infty$. A $\mathfrak{F}_{(\phi,\alpha)}$-segment of lenghth $T$ is a sequence $\mathbf{D}=\{D_i\}_{i=0}^{i=T}$ such that $f(D_i)=D_{i+1}$, each $D_i$ contained in a member of $\mathfrak{F}_{(\phi,\alpha)}$ and $D_T$ is also a member of $\mathfrak{F}_{(\phi,\alpha)}$. $D_0$ and $D_T$ are respectively called the entrance and the exit of $\mathbf{D}$.
\item ($\beta$-controling). For $\beta>0$ and $t\leq T$, $\mathbf{D}$ is said to be $(\beta,t,T,\alpha)$ controlled if there is a set of control times $\mathcal{P}\subset \{0,1,\ldots,T\}$ such that
\begin{itemize}
\item $0,T\in\mathcal{P}$;
\item if $k<l$ and two consecutive elements of $\mathcal{P}$  then $l-k<t$ and $\frac{1}{l-k}|\phi_{l-k}(f^k(x))-\alpha|<\beta$, for all $x\in D_0$.
\end{itemize}
\item (Following $\tau$-pattern). For $\tau>1$, $\mathbf{D}$ follows $\tau$-pattern $s\in\{-,+\}^\mathbb{N}$ up to time time $T$ if for any $n\in [0,T]$ with $n\equiv 0\, (\textrm{mod}\, \tau)$, one has $f^{n+1}(D_0)$ is contained in $\mathfrak{F}^{s_n}_{(\phi,\alpha)}$
\end{itemize}
We can now state the distinctive property of flip-flop family. The following theorem borrows essentially from \cite{bbd}.
\begin{theorem}
Fix any member $D$ of the Flip-flop family $\mathfrak{F}$ and a pattern $s\in\{-,+\}^{\mathbb{N}}$. There are a monotone sequence $\{\beta_i\}$ of positive numbers converging to zero, a monotone sequence $\{t_i\}$ of positive integers converging to infinity, and a sequence of $\mathfrak{F}$-segments ($\mathbf{D}_k$) such that:
\begin{itemize}
\item [\textsc{C1.}] the entrance of $\mathbf{D}_k$ is contained in $D$;
\item [\textsc{C2.}] the length $T_k$ of the segment $\mathbf{D}_k$ goes to infinity with $k$;
\item [\textsc{C3.}] every point in the entrance of $\mathbf{D}_k$ is $(\beta_i,t_i,T_k,\alpha)$-controlled, for every $1\leq i\leq k$.
\item [\textsc{C4.}] $\mathbf{D}_k$ folows $\tau$-ppattern $s$ up to time $T_k$.
\end{itemize}
\end{theorem}
\subsection{Positive Entropy, Second Part of Theorem A}
Now, suppose that $s\in\{-,+\}^\mathbb{N}$ has dense orbit under the shift map. and let $x_k\in\mathbf{D_k}$, $\mathbf{D}_k\subset D\in\mathfrak{F}$, is $(\beta_k,t_k,T_k,\alpha)$-controlledd for $t_k\to\infty$ and $\beta_k\to 0$ provided by $\textsc{C1}$-$\textsc{C4}$. Let $x_k\to x\in D$. It is not difficult to see that $x$ is controlled in any scale and follows $\tau$-pattern $s$ up to infinity. In particular the central Lyapunov exponent of $x$ is equal to $\alpha$.

Now, for any $y\in\omega(x)$, define $\pi(y)\in \{-,+\}^\mathbb{N}$ by
\begin{equation*}
(\pi(y))_i:=\begin{cases}
\begin{array}{ccc}
- & f^i(y)\in \phi^{-1}([\alpha+\chi_0,+\infty]),\\
+ & f^i(y)\in \phi^{-1}([-\infty,\alpha-\chi_0]).
      \end{array}
 \end{cases}
\label{e4}
\end{equation*}
Then $\pi$ is continuous. The mapping $\pi$ is also onto. For the proof, suppose that $\eta=(\eta_{-L},\ldots,\eta_0,\ldots,\eta_L)$ is a finite word of - and +. Since the orbit of $s$ is dense under the shift map, there is a sequence $\ell_i$ with $\ell_i\to\infty$ such that $(s_{\ell_i},\ldots,s_{2L+\ell_i})=\eta$. Let $f^{\tau\ell_i}(x)\to y\in\omega(x)$, then $\pi(y)=\eta$. Since $y\in\omega(x)$, the central Lyapunov exponent of $y$ is also equals to $\alpha$.

Suppose that $\mathcal{K}_\alpha$ be the set all $x\in K$ with the following property: $x\in D\in\mathcal{F}$ is controlled at any scale and follows a $s$-pattern. Put
$$\mathbf{K}_\alpha=\bigcup_{x\in\mathcal{K}_\alpha} \omega(x).$$
Then all point in $\mathbf{K}_\alpha$ has central Lyapunov exponent equal to $\alpha$ and also $h_{\text{top}}(f|_{\mathbf{K}_\alpha})\geq\log 2$. The variational principle for entropy implies the existence of measure of positive entropy supported on the homoclinic class.
\begin{corollary}
The set $\mathbf{K}_\alpha$ is contained in local unstable manifold of $H(p,f)$. In particular, $\mathbf{K}_\alpha$ is contained in $H(p,f)$ provided that $H(p,f)$ is an bi-Lyapunov stable.
\end{corollary}
\section*{Acknowledgements}
During the preparation of this article the author was partially
supported by grant from IPM (No. 95370127). He also thanks ICTP
for supporting through the association schedule.


\begin{thebibliography}{0}
\bibitem{abc}
Abdenur, F., Bonatti, C., and Crovisier, S.~\textit{Non-uniform
hyperbolicity for $C^1$-generic diffeomorphisms}, Israel J. Math., {\bf 183}, (2011), 1-60. 
\bibitem{abcdw}
Abdenur, F., Bonatti, C., Crovisier, C., Diaz, L., and Wen, L.,
~\textit{Periodic points in homoclinic classes}, Ergod. The. and
Dyn. Syst., {\bf 27}, (2007), 1-22.
\bibitem{bp}
Brin, M, and Pesin, Y.,~{\it Partially hyperbolic dynamical
systems}, Izv. Akad. Nauk SSSR Ser. Mat. {\bf 38}, (1974), 170-212.
\bibitem{bdg}
Bonatti, Ch., Diaz, L., and Grodetski, A.,~{\it Non-hyperbolic
ergodic measures with large support}, Nonlinearity, {\bf 23},
(2010), 687-705.
\bibitem{bbd}
Bochi, J., Bonatti, Ch. and Diaz, L. J.\,{\it Robust criterion for the existence of nonhyperbolic
ergodic measures}, Comm. Math. Phys. {\bf 344} (2016), 751–795.
\bibitem{cao}
Cao, Y.,~\textit{Non-zero Lyapunov exponents and uniform
hyperbolicity}, Nonlinearity {\bf 16}, (2003), 1473-1479.
\bibitem{csy}
Crovisier, S., Sambarino, M., and Yang, D.,~{\it Partial
hyperbolicity and homoclinic tangency},  Journal of the European Mathematical Society, {\bf 17}, (2015), 1-49. 
\bibitem{fa}
Fakhari, A.,~{\it Uniform Hyperbolicity Along Periodic Orbits},
Proc. Amer. Math. Soc., {\bf 141}, 9, (2013), 3107-3118
\bibitem{gt}
Gogolev, A. and Tahzibi, A.,\,{\it Center Lyapunov exponents in partially hyperbolic dynamics}, Journal
of Modern Dynamics, {\bf 8}, (2014), 549-576.
\bibitem{gi}
Gorodetski, A.S, and Yu. S. Ilyashenko,~\textit{Certain properties
of skew products over a horseshoe and a solenoid}, Grigorchuk, R. I.
(ed.), Dynamical systems, automata, and infinite groups. Proc.
Steklov Inst. Math. 231, (2000), 90-112.
\bibitem{gikn}
Gorodetski, A.S., Ilyashenko, Yu., Kleptsyn, V. and Nalsky, M.,~{\it
Nonremovable zero Lyapunov esponents}, Functional Analysis and Its
Applications, {\bf 39}, (2005),  27-38.
\bibitem{ga}
Gourmelon, N.,~{\it Th\'ese de Doctorat}, Instabilit\'e de la
dynamique en l\'absence de d\'ecomposition domin\'ee  Universit\'e
de Bourgogne.
\bibitem{hp}
Hasselblatt, H. and Pesin, Y.~\textit{Partially hyperbolic dynamical
systems}, Handbook of dynamical systems. 1B, 1-55, Elsevier B. V.,
Amsterdam, (2006).
\bibitem{hhu}
Hertz, F.R., Hertz, M.A, and Ures, R.,~{\it A survey on partially
hyperbolic dynamics}, 28$^{th}$ Coloquio Brasileiro de Matematica, IMPA Mathematical Publications, IMPA-Rio de Janeiro, (2011), 132.
\bibitem{kn}
Kleptsyn, V. and Nalsky, M.,~{\it Stability of existence of non-hyperbolic measures for $C^1$-diffeomorphisms},
Functional Analysis and Its Applications, (2007), 41:4, 271-283
\bibitem{lor}
Leplaideur, R., Oliveira, K. and Rios, I.~{\it Equilibrium states
for partially hyperbolic horseshoes}, Ergod. Th. \& Dynam. Sys.
(2011), {\bf 31}, 179-195.
\bibitem{pl}
Pliss, V.,~{\it On a conjecture due to Smale}, Diff. Uravnenija.,
{\bf 8}, (1972), 268-282.
\bibitem{ps}
Pujals, E. and Sambarino, M.,~{\it Homoclinic tangencies and
hyperbolicity for surface diffeomorphisms}, Ann. Math., {\bf 151},
(2000), 961-1023.
\bibitem{si}
Sigmund, K.,\,{\it On the connectedness of ergodic systems}, Manuscripta Math. {\bf 22}, (1977), 27-32.
\bibitem{vy}
Viana, M., and Yang, J., {\it Physical measures and absolute continuity for one-dimensional center direction}, Annales Inst. Henri Poincaré-Analyse Non-Linéaire {\bf 30}, (2013), 845-877.
\end{thebibliography}
\end{document}